\documentclass[reqno,b5paper]{amsart}
\usepackage{amsmath}
\usepackage{amssymb}
\usepackage{amsthm}
\usepackage{enumerate}
\usepackage[mathscr]{eucal}
\newenvironment{newlist}
   {\begin{list}{}{\setlength{\labelsep}{0.25cm}
                   \setlength{\labelwidth}{0.65cm}
                   \setlength{\leftmargin}{0.9cm}}}
   {\end{list}}

\newtheorem{df}{Definition}[section]
\newtheorem{thm}[df]{Theorem}
\newtheorem{prop}[df]{Proposition}
\newtheorem{cor}[df]{Corollary}

\newtheorem*{pr1}{Problem 1}

\newtheorem{rem}[df]{Remark}

\def\mt{t\kern-0.035cm\char39\kern-0.03cm}
\def\ml{l\kern-0.035cm\char39\kern-0.03cm}
\def\md{d\kern-0.035cm\char39\kern-0.03cm}
\def\mL{L\kern-0.08cm\char39}
\def\Con{\operatorname{Con}}

\def \c {^{\circ}}
\def \cc{^{\circ\circ}}

\usepackage{pgf,tikz}

\usetikzlibrary{calc,positioning,shapes,arrows.meta}

\tikzset{%
 shaded/.style={draw, shape=circle, fill=black!35, inner sep=1.4pt},
 unshaded/.style={draw, shape=circle, fill=white, inner sep=1.4pt},
 quasi/.style={draw, shape=rectangle, rounded corners=3pt, fill=white, inner sep=2.5pt, minimum height=14.5pt},
 blob/.style={draw, shape=rectangle, rounded corners=12pt, thin, densely dotted},
 arrow/.style={->, thin, >=latex, shorten >=2.5pt, shorten <=2.5pt},
 order/.style={thin},
 curvy/.style={thin, looseness=1.2, bend angle=70},
 fatcurvy/.style={thin, looseness=1.7, bend angle=75},
 label/.style={shape=rectangle, inner sep=6pt},
 auto}

\begin{document}

\title[Congruence pairs of principal MS-algebras]{Congruence pairs of principal MS-algebras and perfect extensions}

\dedicatory{Dedicated to the memory of Professor Beloslav Rie\v
can*}

\author{Abd El-Mohsen Badawy, Miroslav Haviar \and Miroslav Plo\v s\v cica}
\newcommand{\acr}{\newline\indent}

\address{Department of Mathematics\acr Faculty of Science\acr Tanta University\acr Egypt}
\email{abdel-mohsen.mohamed@science.tanta.edu.eg}
\address{Department of Mathematics\acr Faculty of Natural Sciences \acr Matej Bel University\acr Tajovsk\'eho 40, 974 01 Bansk\'a Bystrica\acr Slovakia}
\email{miroslav.haviar@umb.sk}
\address{Institute of Mathematics\acr  Faculty of Natural Sciences \acr  \v Saf\'arik's University \acr  Jesenn\'a 5, 041 54 Ko\v sice\acr Slovakia}
\email{miroslav.ploscica@upjs.sk}

\thanks{The second author acknowledges support from Slovak grant VEGA 1/0337/16. The third author has been supported by VEGA grant 1/0097/18}

\thanks{*See Remark~\ref{Belo}}



\begin{abstract}
The notion of a congruence pair for principal MS-algebras,  simpler than the
one given by Beazer for $K_2$-algebras~\cite{B85}, is introduced. It
is proved that the congruences of the principal MS-algebras $L$
correspond to the MS-congruence pairs on simpler substructures
$L^{\cc}$ and $D(L)$ of $L$ that were associated to~$L$
in~\cite{BGH11}.

An analogy of a well-known Gr\"atzer's problem \cite[Problem
57]{G71} formulated for distributive p-algebras, which asks for a
characterization of the congruence lattices in terms of the
congruence pairs, is presented here for the principal  MS-algebras
(Problem 1). Unlike a recent solution to such a  problem for the
principal p-algebras in~\cite{B17a}, it is demonstrated here on the
class  of principal MS-algebras, that a possible solution to the
problem, though not very descriptive, can be simple and elegant.

As a step to a more descriptive solution of Problem 1, a special
case is then considered  when a principal MS-algebra $L$ is a
perfect extension of its greatest Stone subalgebra $L_{S}$. It is
shown that this is exactly when de Morgan subalgebra $L^{\cc}$ of
$L$ is a perfect extension of the Boolean algebra $B(L)$. Two
examples illustrating when this special case happens and when it
does not are presented.

\end{abstract}

\maketitle

\section{Introduction}\label{Intro}
Blyth and Varlet introduced MS-algebras abstracting de Morgan and
Stone algebras in~\cite{BV1} and \cite{BV2}. In \cite{BGH11} a class
of principal MS-algebras was introduced and a simple triple
construction of the principal MS-algebras was presented. This was
motivated by the second author constructions in~\cite{H94}
and~\cite{H95}. A~one-to-one correspondence between the class of
principle MS-algebras and the class of principle MS-triples was
proved  in \cite{BGH11}. For recent studies of (principle)
MS-algebras see also~\cite{B17c} and~\cite{BA19}.

In Section~\ref{Prel} of this paper we present some properties of
the principal MS-algebras by means of principal MS-triples. Then in
Section~\ref{Pairs} we introduce the notion of congruence pairs in
the principal MS-algebras which is simpler than the notion of
congruence pairs given by Beazer in \cite{B85} for MS-algebras from
the subvariety $K_2$. We prove that every congruence $\theta$ on a
principal MS-algebra $L$ determines so-called MS-congruence pair on
a triple associated to $L$ and conversely, each MS-congruence pair
$(\theta_1, \theta_2)$ of a principle MS-triple uniquely determines
a congruence $\theta$ on the principle MS-algebra $L$ associated to
it.

Our work in this paper has mainly been motivated by a well-known
Gr\"atzer's problem \cite[Problem 57]{G71}, which was formulated for
distributive p-algebras as follows:

{\it Let $B$ be a Boolean algebra, let $D$ be a distributive lattice
with unit, and let $A$ be a sublattice of $\Con(B)\times \Con(D)$.
Under what conditions does there exist a distributive lattice with
pseudocomplementation $L$ such that $B(L)\cong B$, $D(L)\cong D$,
and $A$ consists of all congruence pairs of $L$?}

In \cite[Corollary 1]{K76} Katri\v n\'ak solved this problem and he
in fact characterized, in terms of congruence pairs, the congruence
lattice of any modular $p$-algebra (cf. \cite[Theorem 1]{K76}). This
characterization has been then generalised to quasi-modular
$p$-algebras by El-Assar \cite{A85}. An analogy of the
Gr\"atzer's problem for the principal p-algebras has been considered
by the first author in~\cite{B17a}.

In Section~\ref{GP} we firstly present an analogy of the Gr\"atzer's
problem for the principal MS-algebras (Problem 1). Compared to the
solution for the principal p-algebras in~\cite{B17a} which aims in a
special case (with bounded lattice $D$) to mimic  Katri\v n\'ak's
solution in the general case (with unbounded $D$), our solution here
for the principal MS-algebras is much simpler and a short and
elegant proof is given.

Then we consider the analogy of the Gr\"atzer's problem for the
principal MS-algebras in a special case when a principal MS-algebra
$L$ is a perfect extension of its greatest Stone subalgebra $L_{S}$.
This case is firstly reduced to the property that a de Morgan
subalgebra $L^{\cc}$ of $L$ is a perfect extension of its Boolean
subalgebra $B(L^{\cc})=B(L)$. In Section~\ref{exam} we illustrate on
two examples when this special case happens and when it does not.

\section{Preliminaries}\label{Prel}
An \emph{Ockham algebra} is an algebra $(L;\vee,\wedge,f,0,1)$ of
type $(2,2,1,0,0)$ where $(L; \vee,\wedge, 0, 1)$ is a bounded
distributive lattice and $f$ is a unary operation such that $f(0) =
1$, $f(1) = 0$ and for all $x,y \in L$,
$$
f(x\wedge y) = f(x) \vee f(y) \quad\And \quad f(x \vee y) = f(x)
\wedge f(y).
$$
For an Ockham algebra $(L;\vee,\wedge,f,0,1)$, the subset
$$
S(L) := \{f(x) \mid x \in L\}
$$
is a subalgebra of $L$ which is called the \emph{skeleton} of $L$.

By a \emph{de Morgan-Stone algebra} or briefly an \emph{MS-algebra}
is meant an algebra $(L; \vee, \wedge, ^{\c},0, 1)$ of type
$(2,2,1,0,0)$ where $(L; \vee, \wedge, 0, 1)$ is a bounded
distributive lattice and $^{\c}$ is a unary operation such that for
all $x,y \in L$
\begin{itemize}
    \item[(MS1)] $x\leq x^{\cc}$;
    \item[(MS2)] $(x\wedge y)^{\c}=x^{\c}\vee y^{\c}$;
    \item[(MS3)] $1^{\c}=0$.
\end{itemize}

The classes of all Ockham algebras and of all MS-algebras are
equational, i.e. form varieties of algebras. Their relationship is
described in \cite[Theorem 1.4]{BV3} as follows: \emph{Every
MS-algebra is an Ockham algebra with a de Morgan skeleton. An Ockham
algebra $(L;\vee,\wedge,f,0,1)$ is an MS-algebra if and only if $x
\leq f^2(x)$ for all $x \in L$.}

Among important subvarieties of MS-algebras is the class of \emph{de
Morgan algebras} that satisfy the additional identity
\begin{itemize}
    \item[(MS4)] $x=x^{\cc}$.
\end{itemize}
Another important subvariety is that of \emph{Kleene algebras} which
are de Morgan algebras satisfying the identity
\begin{itemize}
    \item[(MS5)] $(x \wedge x^{\c}) \vee y \vee y^{\c} = y \vee y^{\c}$.
\end{itemize}
The subvariety of \emph{Stone algebras} is characterized by the
identity $x \wedge x^{\c} = 0$. The subvariety of
\emph{$K_2$-algebras} is characterized by the identity (MS5)
together with the identity
\begin{itemize}
    \item[(MS6)] $x \wedge x^{\c} = x^{\cc} \wedge x^{\c}$.
\end{itemize}
Finally, the subvariety of MS-algebras characterized by the identity
$x \vee x^{\c} = 1$ is the well-known class of \emph{Boolean
algebras}.

For an MS-algebra $L$, important subsets playing a  role in its
investigation are:
\begin{itemize}
    \item[(i)] the subalgebra $L^{\cc} = \{x\in L \mid x = x^{\cc}\}$  of
    \emph{closed} elements of $L$ which is  a de Morgan algebra (this is the skeleton $S(L)$ of $L$ when $L$
    is considered as an Ockham algebra);
    \item[(ii)] the filter $D(L) = \{x\in L \mid x^{\c}=0\}$ of  \emph{dense} elements of
    $L$;
  \item[(iii)] the Boolean subalgebra $B(L) = \{x\in L \mid x \vee x^{\c} = 1\} $ of \emph{complemented} elements of $L$;
  \item[(iv)]  the generalisation of  $B(L)$, the subalgebra $L_S = \{x\in L \mid x^{\c} \vee  x^{\cc} = 1\}$ of the elements of $L$ satisfying the \emph{Stone
  identity} which is obviously the greatest Stone subalgebra of $L$.
\end{itemize}

 It is easy to see that  $B(L)= B(L^{\cc})= B(L_S)= L_{S}^{\cc}$ and $D(L)=D(L_S)$. If $L$ is a Stone algebra itself, then clearly $B(L) = L^{\cc}$ and $L_S =
 L$. It is also obvious that if $L$ is an MS-algebra with a smallest dense element $d_L$, then the map $\varphi(L): L^{\cc} \to D(L)$ defined by
$\varphi(L)(a) = a\vee d_L$ is a $(0,1)$-lattice homomorphism.

We recall the following definitions from \cite{BGH11} motivated by the work in~\cite{H95}.

\begin{df}\label{pr}
An MS-algebra $(L; \vee,\wedge,^{\c},0,1)$ is called a {principal
MS-algebra} if it satisfies the following conditions:
\begin{itemize}
    \item[(i)] The filter $D(L)$ is principal, i.e. there exists an element $d_L \in L$ such that $D(L)=\left[d_L\right)$;
    \item[(ii)] $x=x^{\cc}\wedge (x\vee d_L)$ for any $x\in L$.
\end{itemize}
\end{df}

\begin{df}
An \rm{(}abstract\rm{)} {principal MS-triple} is $(M,D, \varphi)$, where
\begin{itemize}
    \item[(i)] $M$ is a de Morgan algebra;
    \item[(ii)] $D$ is a bounded distributive lattice;
    \item[(iii)] $\varphi$ is a $(0,1)$-lattice homomorphism from $M$ into $D$.
\end{itemize}
\end{df}

In \cite{BGH11} we proved the following result.

\begin{thm}\label{PC}
Let $(M,D,\varphi)$ be a principal
MS-triple. Then
$$
L= \{(x,y) \mid x\in M, y \in D, y \leq \varphi(x)\}
$$
is a principal MS-algebra, if one defines
$$
(x_1,y_1)\vee (x_2,y_2)=(x_1\vee x_2, y_1\vee y_2)
$$
$$
(x_1,y_1)\wedge (x_2,y_2)=(x_1\wedge x_2, y_1\wedge y_2)
$$
$$
(x,y)^{\c}=(x^{\c},\varphi(x^{\c}))
$$
$$
1_L=(1,1)
$$
$$
0_L=(0,0).
$$
Also $L^{\cc}=\{(x,\varphi(x))\mid x\in M\} \cong M$ and $D(L)=\{(1_M,y)\mid y\in D\} \cong D$.
\end{thm}

In \cite{H94} the second author introduced a concept of so-called $K_2$-triple. By a \emph{principal $K_2$-triple} is meant
a triple $(K, D, \varphi)$ in which $K$ is a Kleene algebra and $\varphi(K^{\wedge}) = \{0_D\}$. In \cite[Corollary 1]{H94} a restriction of
Theorem~\ref{PC} was proven for the principal $K_2$-triples which says that if the principal
MS-triple $(M,D,\varphi)$ is a principal $K_2$-triple, then the MS-algebra $L$ from Theorem~\ref{PC} is a principal $K_2$-algebra, that
is, a $K_2$-algebra with a principal filter $D(L)$. We now make a small addition to this result.

While Theorem~\ref{PC} shows that the filter $D(L)$ of the principal MS-algebra $L$ associated with a principal MS-triple $(M, D, \varphi)$
is isomorphic to $D$, we can consider and describe the larger filter $L^{\vee} = \{x\vee x^{\c}\mid x\in L\}$ of the principal $K_2$-algebra $L$ associated
with a principal $K_2$-triple $(K, D, \varphi)$. It is evident that  $L^{\vee} \supseteq D(L)$. And dually, we can describe the ideal
$L^{\wedge} = \{x\wedge x^{\c}\mid x\in L\}$ of the principal $K_2$-algebra $L$.

\begin{prop}\label{prop1}
Let $(K, D, \varphi)$ be a principal $K_2$-triple. Then for the
principal $K_2$-algebra $L$ obtained from it by the construction
given in Theorem~\ref{PC}, we have the following descriptions of its
filter $L^{\vee}$ and the ideal $L^{\wedge}$:
\begin{align}
&L^{\vee}= \{(x,y)\in L \mid x \in K^{\vee}\},\notag\\
&L^{\wedge}= \{(x,y)\in L \mid x \in K^{\wedge}\}.\notag
\end{align}
\end{prop}

\begin{proof}
Let $(K, D, \varphi)$ be a principal $K_2$-triple. Using the
construction in Theorem~\ref{PC} we obtain  a principal
$K_2$-algebra $L = \{(x,y) \mid x\in K, y\in D, y \leq
\varphi(x)\}$. Now
\begin{align}
L^{\vee} &= \{(x,y)\vee (x,y)^{\c} \mid (x,y) \in L\} \notag \\
&=\{(x\vee x^{\c}, y \vee \varphi(x^{\c})) \mid x\in K, y\in D, y \leq \varphi(x)\} \notag \\
&=\{((x\vee x^{\c},y')  \mid x  \in K, y'\in D, y' \leq \varphi(x\vee x^{\c})\} \notag \\
&=\{(x',y') \in L \mid x' \in K^{\vee}\}. \notag
\end{align}
Dually one can derive the description of $L^{\wedge}$.
\end{proof}

By a principal $S$-triple we mean a principal $K_2$-triple $(B, D,
\varphi)$ where $B$ is a Boolean algebra.

\begin{cor}
Let $(B, D, \varphi)$ be a principal $S$-triple. Then the 
principal $K_2$-algebra $L$ from Theorem~\ref{PC} is a principal
Stone algebra with $L^{\vee}= D(L)$ and $L^{\wedge}= \{0_L\}$.
\end{cor}

\begin{proof}
Let $(x,y)\in L$. Then we have
\begin{align}
(x,y)^{\cc} \vee (x,y)^{\c}&=(x,\varphi(x)) \vee (x^{\c}, \varphi(x^{\c})) \notag \\
&=(x \vee x^{\c},\varphi(x \vee x^{\c})) \notag \\
&=(1_M, 1_D), \notag
\end{align}
as $x\vee x^{\c} =1$. Hence $L$ is a Stone algebra. By
Proposition~\ref{prop1} we obtain
\begin{align}
L^{\vee} &= \{(x,y) \in L \mid x\in B^{\vee} \} \notag \\
&= \{(1_B,y) \in L \mid y\in D \} \notag \\
&=D(L). \notag
\end{align}
and
\begin{align}
L^{\wedge} &= \{(x,y) \in L \mid x\in B^{\wedge}\} \notag \\
 &= \{(0_B,y) \mid y\in D, y\le \varphi (0_B)\} \notag \\
&=\{(0_B, 0_D)\} 
= \{0_L\}.\notag
\end{align}
\end{proof}

\section{Congruence pairs of principal MS-algebras}\label{Pairs}

Let $L$ be a principal MS-algebra with a smallest dense element
$d_L$. For a congruence relation $\theta$ of $L$, let
$\theta_{L^{\cc}}$ and $\theta_{D(L)}$ denote the restrictions of
$\theta$ to $L^{\cc}$ and $D(L)=\left[ d_{L}\right)$, respectively.
Let $\Con(L^{\cc})$ and $\Con(D(L))$ denote the congruence lattices
of the de Morgan algebra $L^{\cc}$ and of the distributive lattice
$D(L)$, respectively.
 Obviously,
$(\theta_{L^{\cc}}, \theta_{D(L)}) \in \Con(L^{\cc}) \times
\Con(D(L))$.

\begin{df}\label{CP}
Let $L$ be a principal MS-algebra with a smallest dense element
$d_L$. A pair of congruences $(\theta_1, \theta_2) \in \Con(L^{\cc})
\times \Con(D(L))$ will be called an MS-congruence pair if
\begin{newlist}
\item[\rm(CP)]
$(a,b) \in \theta_1 \ \text{implies} \ (a \vee d_L, b\vee d_L) \in
\theta_2$.
\end{newlist}
\end{df}

We notice that in \cite{B85} R. Beazer defined the notion of a
congruence pair for the MS-algebras from the subvariety $K_2$ by two
conditions. More precisely, for a  $K_2$-algebra $L$ he called a
pair of congruences $(\theta_1, \theta_2) \in \Con(L^{\cc}) \times
\Con(L^\vee)$ a \emph{$K_2$-congruence pair} if
\begin{newlist}
\item[\rm(CP$_1$)] $(c,d) \in \theta_2 \ \text{implies} \ (c^{\c},d^{\c}) \in \theta_1$;
\item[\rm(CP$_2$)] $(a,b) \in \theta_1 \ \text{and}\ c\in L^\vee\ \text{imply} \ (a \vee c, b\vee c) \in \theta_2$.
\end{newlist}

One can show that for our MS-congruence pairs $(\theta_1,
\theta_2)$, the Beazer first condition (CP$_1$) is trivially
satisfied and so it can be removed from our definition: the reason
is that we consider the congruence $\theta_2$ on a smaller filter
$D(L)\subseteq L^\vee$.  It is also clear that for the principal
MS-algebras, the Beazer second condition (CP$_2$) restricted to the
filter $D(L)$ implies our condition (CP) since his condition is
quantified for all elements $c$ of the filter. Yet our simpler
condition (CP) is  equivalent to (CP$_2$) restricted to $D(L)$
because if $(\theta_1, \theta_2)\in \Con(L^{\cc}) \times \Con(D(L))$
is an MS-congruence pair, then $(a,b) \in \theta_1$
 implies $(a \vee d_L, b\vee d_L) \in \theta_2$, whence
obviously $(a\vee c,b\vee c) = (a \vee d_L\vee c, b\vee d_L\vee c)
\in \theta_2$ for all $c\in D(L)$.

Now we apply our definition of the MS-congruence pair to show that
the congruences of principal MS-algebras correspond to the
MS-congruence pairs.

\begin{thm}\label{thm3}
Let $L$ be a principal MS-algebra with a smallest dense element
$d_L$. For every congruence $\theta$ on $L$, the restrictions of
$\theta$ to $L^{\cc}$ and $D(L)=\left[ d_{L}\right)$ determine the
MS-congruence pair $(\theta_{L^{\cc}}, \theta_{D(L)})$.

Conversely, every MS-congruence pair $(\theta_1, \theta_2)$
uniquely determines a congruence $\theta$ on $L$ satisfying
 $\theta_{L^{\cc}} = \theta_1$ and $\theta_{D(L)} = \theta_2$. This congruence can be defined by the rule
$$
(x,y)\in \theta \quad  \text{if} \quad  (x^{\c}, y^{\c}) \in \theta_1 \
\text{and} \ (x \vee d_L, y\vee d_L) \in \theta_2.
$$
\end{thm}

\begin{proof}
For every $\theta\in \Con(L)$, the restrictions of
$\theta$ to $L^{\cc}$ and $\left[ d_{L}\right)$ determine the
MS-congruence pair $(\theta_{L^{\cc}}, \theta_{D(L)})$
because  $(a,b) \in \theta_{L^{\cc}}$ gives $(a \vee d_L,
b\vee d_L) \in \theta$, thus $(a \vee d_L, b\vee d_L) \in
\theta_{D(L)}$.

Now let $(\theta_1, \theta_2)\in \Con(L^{\cc}) \times \Con(D(L))$ be
an MS-congruence pair and $\theta$ be the relation on $L$ defined by
the above rule. Clearly $\theta$ is an equivalence. To show that
$\theta$ is a congruence of $L$, let $(a,b) \in \theta$ and $(c,d)
\in \theta$. Then by the definition of~$\theta$, $(a^{\c},b^{\c}),
(c^{\c},d^{\c}) \in \theta_1 $ and $(a\vee d_L, b\vee d_L), (c\vee
d_L, d\vee d_L)\in \theta_2$. As $\theta_1\in \Con(L^{\cc})$, we get
$(a^{\c} \vee c^{\c}, b^{\c} \vee d^{\c}) \in \theta_1$, so $((a
\wedge c)^{\c}, (b \wedge d)^{\c}) \in~\theta_1$. Using the
distributivity of $L$  we have $((a \wedge c)\vee d_L, (b\wedge d)
\vee d_L) \in \theta_2$. This shows that $(a \wedge c, b\wedge d)
\in \theta$, so $\theta$ is preserved by the meet operation. To show
that $\theta$ is preserved by the join operation is immediate as the
congruence $\theta_1$ of $L^{\cc}$ contains with the pairs
$(a^{\c},b^{\c}), (c^{\c},d^{\c})$ also the pair $(a^{\c} \wedge
c^{\c},b^{\c} \wedge d^{\c})$, so $((a\vee c)^{\c}, (b\vee d)^{\c})
\in \theta_1$. And clearly $((a\vee c) \vee d_L, (b\vee d) \vee
d_L)\in \theta_2$. Hence $(a \vee c, b\vee d) \in \theta$.

Now let $(a,b)\in \theta$. Then $(a^{\c}, b^{\c})\in \theta_1$ and
$(a^{\cc}, b^{\cc}) \in \theta_1$. From $(a^{\c}, b^{\c})\in\theta_1$ we
have $(a^{\c}\vee d_L, b^{\c} \vee d_L)\in \theta_2$ by
Definition~\ref{CP}. Thus $(a^{\c}, b^{\c}) \in \theta$.

Next we will show that $\theta_{L^{\cc}} = \theta_1$ and
$\theta_{D(L)} = \theta_2$. To show $\theta_1 \leq
\theta_{L^{\cc}}$, let $a,b \in L^{\cc}$ and $(a,b) \in \theta_1$.
Then $(a^{\c}, b^{\c})\in \theta_1$ so $(a \vee d_L, b\vee d_L) \in
\theta_2$ as $(\theta_1, \theta_2)$ is an MS-congruence pair. Hence
$(a,b)\in \theta$ and $(a,b) \in \theta_{L^{\cc}}$. Conversely, let
$(a,b) \in \theta_{L^{\cc}}$, so $(a^{\c}, b^{\c}) \in \theta_1$.
Then $(a^{\cc}, b^{\cc}) \in \theta_1$, whence $(a,b) \in \theta_1$.

Now let $c,d \in D(L)$ and $(c,d) \in \theta_2$. Then $(c\vee d_L,
d\vee d_L)\in \theta_2$ and using the fact that $c^{\c} = d^{\c} = 0$ we
get $(c,d)\in \theta$, which means $(c,d) \in \theta_{D(L)}$.
Conversely, let $(c,d) \in \theta_{D(L)}$. Then $(c,d)\in \theta$
and since $c= c\vee d_L, d = d \vee d_L$, we obtain $(c,d)\in
\theta_2$. Hence $\theta_{D(L)} = \theta_2$.

To show the uniqueness of $\theta$, let $\theta$ and $\theta'$ be
congruences on $L$ such that $\theta_{L^{\cc}} =\theta'_{L^{\cc}} =
\theta_1$ and $\theta_{D(L)} = \theta'_{D(L)}= \theta_2$. Then
$(x,y) \in \theta$ gives $(x^{\cc}, y^{\cc}) \in \theta_{L^{\cc}}$ and
$(x\vee d_L, y \vee d_L) \in \theta_{D(L)}$ and also $(x^{\cc},
y^{\cc}) \in \theta'_{L^{\cc}}$ and $(x\vee d_L, y \vee d_L) \in
\theta'_{D(L)}$. Therefore $(x^{\cc} \wedge(x \vee d_L), y^{\cc}\wedge
(y\vee d_L))\in \theta'$. As $L$ is a principal MS-algebra with a
smallest dense element $d_L$, we get $(x, y) \in \theta'$. The
inclusion $\theta' \subseteq \theta$ can be shown analogously. Hence
$\theta = \theta'$.
\end{proof}

The set  of all MS-congruence pairs of $L$ will be denoted by $A(L)$ which is a traditional notation for the congruence pairs of algebras in the literature.

\begin{cor}
Let $L$ be a principal MS-algebra with a smallest dense element
$d_L$. The set $A(L)$ of MS-congruence pairs of $L$ forms a
sublattice of $\Con(L^{\cc}) \times \Con(D(L))$ and $\theta
\longmapsto (\theta_{L^{\cc}}, \theta_{D(L)})$ is an isomorphism
between $\Con(L)$ and $A(L)$.
\end{cor}

\begin{proof}
Take $(\phi_1,\psi_1), (\phi_2, \psi_2) \in A(L)$. It is evident
that $(\phi_1 \wedge \phi_2, \psi_1 \wedge \psi_2) \in A(L)$. To
prove that $(\theta_1 \vee \psi_1, \theta_2 \vee \psi_2) \in A(L)$,
let $(a,b) \in \theta_1 \vee \psi_1$. Then there exists a finite
sequence $a = a_0, a_1, \ldots, a_n = b$ in $L^{\cc}$ with
$(a_{i-1}, a_i) \in \theta_1 \vee \psi_1$ for $i\in \{1, \ldots,
n\}$. By Definition~\ref{CP} we have $(a_{i-1}\vee d_L, a_i \vee
d_L) \in \theta_2 \vee \psi_2$. Hence the sequence
$$
a\vee d_L = a_0 \vee d_L, a_1 \vee d_L, \ldots, a_n \vee d_L = b\vee
d_L
$$
in $D(L)$ is witnessing $(a \vee d_L, b \vee d_L) \in \theta_2 \vee
\psi_2$. Thus $(\theta_1 \vee \psi_1, \theta_2 \vee \psi_2) \in
A(L)$ showing that $A(L)$ is a sublattice of $\Con(L^{\cc}) \times
\Con(D(L))$. From Theorem~\ref{thm3} it follows that the map
$\theta \longmapsto (\theta_{L^{\cc}}, \theta_{D(L)})$ is an
isomorphism.
\end{proof}

\section{\textbf Gr\"atzer's Problem and perfect extensions}\label{GP}

In this section we compare our study with that for (distributive)
p-algebras. For the following basic facts about p-algebras see
\cite{K80}. We recall that a $p$-algebra is an algebra
$(L,\vee,\wedge,^{\ast},0,1)$ with a bounded lattice reduct
$(L,\vee,\wedge,0,1)$ such that the unary operation of
pseudocomplementation $^{\ast}$ is defined on $L$ by $x\wedge a = 0$
if and only if  $x \leq a^\ast.$

The first proof that the class of all $p$-algebras is equational is
due to P.~Ribenboim  \cite{R49}. A $p$-algebra $L$ is said to be
distributive (modular) if the underlying lattice
$(L,\vee,\wedge,0,1)$ is distributive (modular). Further, a
$p$-algebra $L$ is called quasi-modular if it satisfies the identity
$((x\wedge y)\vee z^{\ast\ast})\wedge x=(x\wedge
y)\vee(z^{\ast\ast}\wedge x)$. It follows that it also satisfies the
identity $x=x^{\ast\ast}\wedge(x\vee x^{\ast})$. The class of all
quasi-modular p-algebras contains the class of all modular
$p$-algebras.  Finally, if in a $p$-algebra $L$ the Stone identity $
x^{\ast}\vee x^{\ast\ast}=1 $ is satisfied, then it is said to be an
$S$-algebra. In the distributive case the $S$-algebras coincide with
the \emph{Stone algebras}.

For a $p$-algebra $L$ the subset $B(L)= \{a\in L \mid
a=a^{\ast\ast}\}$ of \emph{closed} elements is a Boolean algebra
$(B(L),\sqcup, \wedge, 0, 1)$
 with $a\sqcup b =(a^{\ast}\wedge b^{\ast})^{\ast}$ and the subset $D(L)= \{a\in L \mid
d^{\ast}=0\}$ of \emph{dense} elements of $L$ is a
filter of $L$.

Gr\"atzer's problem from his book~\cite{G71} (Problem 57) for
distributive p-algebras was formulated at the end of
Section~\ref{Intro}.

We present an analogy of Gr\"atzer's problem  for the principal
MS-algebras here. However, unlike Gr\"atzer's formulation of
\cite[Problem 57]{G71} for distributive p-algebras and its solution
by Katri\v n\'ak in \cite{K76}, we do not apply the traditional
approach of identifying isomorphic objects: in Gr\"atzer's problem
and Katri\v n\'ak's solution these were $B(L)$ and $B$, $D(L)$ and
$D$, and the sets $A$ and $A(L)$ of congruence pairs. In our
approach below we take into account the isomorphisms between them.

For a homomorphism of  algebras
$\tau: A \to B$ and a congruence $\theta \in \Con (A)$, the
congruence on $B$ generated by all pairs $(\tau (x), \tau (y))$ with
$(x, y)\in \theta$ will be denoted by $\Con
(\tau)(\theta)$. It is well known (and easy to show) that the
mapping $\Con(\tau): \Con(A) \to \Con(B)$ preserves arbitrary joins.

Now let $M$ be a de Morgan algebra, let $D$ be a bounded
distributive lattice and let $A$ be a subset of $\Con(M)\times
\Con(D)$. To represent $A$ as the set $A(L)$ of the congruence pairs of some principal
$MS$-algebra $L$, we seek for $L$ and isomorphisms $\tau_1: M\to
L^{\cc}$ and $\tau_2: D\to D(L)$ such that
\[
A=\{(\theta_{1},\theta_2) \mid (\Con (\tau_1)(\theta_1),\Con(\tau_2)(\theta_2))\in A(L)\} \tag{*}.
\]
Then $A\subseteq \Con(M)\times \Con(D)$ is called \emph{representable}.

Now we present
an analogy of Gr\"atzer's problem  for the principal
MS-algebras.

\begin{pr1}\label{prob1}
Let $M$ be a de Morgan algebra, let $D$ be a bounded distributive
lattice, and let $A$ be a sublattice of $\Con(M)\times \Con(D)$.
Under what conditions does there exist a principal MS-algebra $L$
and isomorphisms $\tau_1: M\to L^{\cc}$ and $\tau_2: D\to D(L)$ such that $A$ is representable?
\end{pr1}

A basic answer is the following.

\begin{thm}\label{thm:prob0}
Let $M$ be a de Morgan algebra, let $D$ be a bounded distributive
lattice, and let $A$ be a sublattice of $\Con(M)\times \Con(D)$.
Then  $A$ is representable if and only if
there exists a lattice $(0,1)$-homomorphism $\varphi: M\to D$ such that for every  $\theta_1\in \Con(M)$ and every $\theta_{2}\in \Con(D)$
\[
(\theta_{1},\theta_2) \in A \iff \Con (\varphi)(\theta_1) \subseteq \theta_2.\tag{**}
\]
\end{thm}

\begin{proof} First, let $A$ be representable. So, there exists a principal MS-algebra $L$
and isomorphisms $\tau_1:\ M\to L^{\cc}$ and $\tau_2:\ D\to D(L)$ such that (*) is satisfied.
We can assume that  the MS-algebra $L$ is given by the triple $(M,D,\psi)$, so
$L^{\cc}=\{(x,\psi(x))\mid x\in M\}$, $D(L)=\{(1_M,y)\mid y\in D\}$, $d_L=(1_M,0_D)$.
The isomorphisms $\tau_1$, $\tau_2$ must have the form
$$\tau_1(x)=(\sigma_1(x),\psi(\sigma_1(x)),$$
$$\tau_2(y)=(1_M,\sigma_2(y)),$$
for some automorphisms $\sigma_1$ and $\sigma_2$ on $M$ and $D$, respectively.
We set $\varphi=\sigma_2^{-1}\psi\sigma_1$. It is not difficult to check (**).

Conversely, let $\varphi:\ M\to D$ satisfy (**). Let $L$ be the MS-algebra determined by the triple
$(M, D,\varphi)$. We define the isomorphisms $\tau_i$ by $\tau_1(x)=(x,\varphi(x))$, $\tau_2(y)=(1_M,y)$.
Then it is not difficult to check that (*) holds.
\end{proof}

Now we present a modification of the above theorem, where only principal congruences on $M$ are needed.

\begin{thm}\label{thm:prob1}
Let $M$ be a de Morgan algebra, let $D$ be a bounded distributive
lattice, and let $A$ be a sublattice of $\Con(M)\times \Con(D)$.

Then  $A$ is representable if and only if the following conditions are satisfied:
\begin{newlist}
\item[\rm(1)] $A$ is join-closed;
\item[\rm(2)] $A$ is down-closed in first coordinate, that is, $(\theta_1,\theta_2)\in A$ and $\alpha\leq\theta_{1}$ imply $(\alpha,\theta_{2})\in A$;
\item[\rm(3)] there exists a lattice $(0,1)$-homomorphism $\varphi: M\to D$ such that for every principal congruence $\theta_1\in \Con(M)$ and every $\theta_{2}\in \Con(D)$
\[
(\theta_{1},\theta_2) \in A \iff \Con (\varphi)(\theta_1) \subseteq \theta_2.
\]
\end{newlist}
\end{thm}

\begin{proof} The necessity of the conditions (1)-(3) for representability of $A$ is clear. For the converse assume that the conditions (1)-(3) are
satisfied. Every congruence $\theta_1\in \Con(M)$ is a join of
principal congruences $\alpha_i\in \Con(M)$ where $i\in I$ for some
set $I$. Since the mapping $\Con(\varphi)$  preserves joins, we
obtain
\begin{align}
(\theta_{1},\theta_2) \in A  &\iff (\forall i \in I)\ (\alpha_i,\theta_2) \in A\notag\\
&\iff (\forall i \in I)\ \Con(\varphi)(\alpha_i) \subseteq \theta_2\notag\\
&\iff \Con (\varphi)(\theta_1) \subseteq \theta_2.\notag
\end{align}
Hence $A$ is representable.
\end{proof}

As a step to achieve a more descriptive solution of Problem 1, for the rest of this section
we consider a special
case when a principal MS-algebra $L$ is a so-called perfect extension of its
greatest Stone subalgebra $L_{S}$.
 In our first theorem below we are
able to reduce such a condition to the usual simpler substructures
$L^{\cc}$ and $D(L)$ of $L$.

It is appropriate to recall here some concepts.  An algebra $A$
satisfies the \emph{Congruence Extension Property} (briefly ({CEP}))
if for every subalgebra $B$ of $A$ and every congruence $\theta$ of
$B$, $\theta$ extends to a congruence of $A$. An algebra $A$ is said
to be a \emph{perfect extension} of its subalgebra $B$, if every
congruence of $B$ has a unique extension to $A$ (see \cite[page
30]{BV3}). We notice that in the literature such $A$ is sometimes
called a \emph{congruence-preserving extension} of $B$ (we e.g.
refer to \cite{GW} by Gr\"atzer and Wehrung where this concept is used in case of lattices).
Of course, if $A$ is a perfect extension of $B$, then $\Con(A) \cong
\Con(B)$.

The classes of distributive lattices and of MS-algebras are known to
satisfy the (CEP) \cite{G71}, \cite{BV3}. However, very little seems
to be known when  in these or other classes of algebras, an algebra
$A$ is a perfect extension of its subalgebra $B$. In \cite{GW},
Gr\"atzer and Wehrung proved that every lattice with more than one
element has a proper congruence-preserving (i.e. perfect) extension. In
\cite[Theorem 2.13]{BV3} it is shown that for any Ockham algebra $O$
with a fixed point $a$, $O$ is a perfect extension of its subalgebra
$C_a$ which is the cone generated by $a$:
$$
C_a := \downarrow a\ \cup \uparrow a = \{x\in O\mid x\leq a\}\cup
\{x\in O\mid x\geq a\}.
$$
Consequently, $\Con(A)\cong \Con(C_a)$.
For an Ockham algebra $O$ its subalgebra $C_a$  always satisfies the
axiom (MS5) and the size of $C_a$ can in general be much smaller
than that of $O$ (see \cite[page 30]{BV3}).

\begin{rem}\label{Belo}
We notice that in \cite{V} it was shown that in a modular lattice $L$, the
cone $C_a$ generated by a (where $a\in L$ is an arbitrary element) is a sublattice with the property
that every congruence of $C_a$ has at most one extension to a congruence of $L$.

The theory of modular lattices was the first research topic of
Beloslav Rie\v can to whose memory this paper is dedicated. Under
the guidance of his young professor Milan Kolibiar, in 1957 as a
student Belo proved that the axiomatic of modular lattices can be
built by using only two axioms together with associativity. The
paper~\cite{R58} where this result was published was the first
scientific paper of professor Beloslav Rie\v can.
\end{rem}

\begin{thm}
Let $L'$ be a subalgebra of a principal MS-algebra $L$. Then $L$ is
a perfect extension of $L'$ if and only if
\begin{newlist}
\item[\rm(1)]  $D(L)$ is a perfect extension of $D(L')$ and
\item[\rm(2)]  $L^{\circ\circ}$ is a perfect extension of
$(L')^{\circ\circ}$.
\end{newlist}
\end{thm}

\begin{proof}
Let $L$ be a perfect extension of $L'$. Let $\psi\in \Con(D(L'))$.
As the (CEP) holds for the class of distributive lattices, we only
have to verify that $\psi$ has a unique extension to a congruence of
$D(L)$. Let $\psi_1,\psi_2\in \Con(D(L))$ such that
$\psi_1\restriction D(L')=\psi_2\restriction D(L')=\psi$. Clearly
$(\triangle_{L^{\circ\circ}},\psi_1),(\triangle_{L^{\circ\circ}},\psi_2)\in
A(L)$. Using Theorem~\ref{thm3}, there exist $\theta_1$ and
$\theta_2$ of $\Con(L)$ corresponding to
$(\triangle_{L^{\circ\circ}},\psi_1)$ and
$(\triangle_{L^{\circ\circ}},\psi_2)$, respectively. Since
$(\triangle_{(L')^{\circ\circ}},\psi_1\restriction D(L')),
(\triangle_{(L')^{\circ\circ}},\psi_2\restriction D(L'))$ in $A(L')$
corresponding to the congruences $\theta_1\restriction  L',
\theta_2\restriction L'\in \Con(L')$ are the same,  we obtain
$\theta_1\restriction  L' = \theta_2\restriction L'$. As $L$ is a
perfect extension of $L'$, it follows $\theta_1=\theta_2$. Hence
$\psi_1=\psi_2$ proving (1). Now we show that $L^{\circ\circ}$ is a
perfect extension of $(L')^{\circ\circ}$. Let $\phi\in
\Con((L')^{\circ\circ})$. Then $\phi$ has an extension to a
congruence of $L^{\circ\circ}$ by the (CEP). To show that this
extension is unique, let $\phi_1,\phi_2\in \Con(L^{\circ\circ})$
with $\phi_1\restriction  (L')^{\circ\circ}=\phi_2\restriction
(L')^{\circ\circ}=\phi$. Clearly $(\phi_1,\nabla_{D(L)}),
(\phi_2,\nabla_{D(L)})\in A(L)$.  As above, by Theorem~\ref{thm3},
there exist $\theta_1$ and $\theta_2$ of $\Con(L)$ corresponding to
$(\phi_1,\nabla_{D(L)})$ and $(\phi_2,\nabla_{D(L)})$, respectively.
Since the pairs $(\phi_1\restriction
(L')^{\circ\circ},\nabla_{D(L)})$ and $(\phi_\restriction
(L')^{\circ\circ}2,\nabla_{D(L)})$  corresponding to the congruences
$\theta_1\restriction  L', \theta_2\restriction L'$ of $L'$ are the
same,  we again obtain $\theta_1\restriction  L' =
\theta_2\restriction L'$.  Because $L$ is a perfect extension of
$L'$, we get $\theta_1=\theta_2$, whence
 $\phi_1=\phi_2$ proving that
$L^{\circ\circ}$ is a perfect extension of $(L')^{\circ\circ}$.

Conversely, let the conditions (1) and (2) hold and let $\theta' \in
\Con(L')$. By the (CEP) for the class of MS-algebras, $\theta'$ has
an extension to a congruence of $L$. To show
 that this extension is unique, assume that $\theta_1$ and
$\theta_2$ are extensions of $\theta'$ in $\Con(L)$.
 By Theorem~\ref{thm3}, the congruences $\theta_1, \theta_2$  can be represented by some congruence pairs $(\phi_1,\psi_1), (\phi_2,\psi_2)\in
 A(L)$.
It is easy to see that it follows
$(\phi_1\restriction\!(L')^{\circ\circ},\psi_1\restriction\! D(L')),
(\phi_2\restriction\! (L')^{\circ\circ},\psi_2\restriction\!
D(L'))\in A(L')$ and these pairs both represent the congruence
$\theta' \in \Con(L')$. Hence we have $\phi_1\restriction
(L')^{\circ\circ} = \phi_2\restriction  (L')^{\circ\circ}$ and
$\psi_1\restriction D(L')=\psi_2\restriction D(L')$. By the
conditions (1) and (2), we obtain $\phi_1=\phi_2$ and
$\psi_1=\psi_2$. Therefore $\theta_1=\theta_2$ as required.
\end{proof}

The following result is now an immediate consequence of the above
theorem and the equalities $B(L)=L_{S}^{\cc}$ and $D(L)=D(L_S)$.

\begin{cor}\label{5.2}
A principal MS-algebra $L$ is a perfect extension of its greatest Stone subalgebra $L_{S}$ if and
only if $L^{\cc}$ is a perfect extension of $B(L)$.
\end{cor}

We notice that the lattice of
MS-congruence pairs of the greatest Stone subalgebra $L_{S}$ of an MS-algebra $L$ is
$A(L_{S})=\{(\phi_{L_{S}^{\circ\circ}},\psi)\in
\Con(B(L))\times \Con(D(L))\mid (\phi,\psi)\in A(L)\}$,
where $\phi_{L_{S}^{\circ\circ}}$ is the restriction of $\phi$ to
$B(L)$.

\begin{cor}\label{5.3}
Let a principal MS-algebra $L$ be a perfect extension of its greatest Stone subalgebra $L_{S}$.
Then $\Con(L)\cong \Con(L_{S})$. 
\end{cor}


Let $L$ be a principal MS-algebra associated to a principal
MS-triple $(M,D,\varphi)$. Then obviously $L_{S}=\{(a,x)\in L\mid
a\in B(M)\}$ is the greatest Stone subalgebra of $L$ such that
$L_{S}^{\cc}=B(L)\cong B(M)$ and $D(L_{S})=D(L)\cong D$. Hence our
final corollary is stated in terms of principal MS-triples.

\begin{cor}\label{5.4}
Let $L$ be a principal MS-algebra associated with a principal
MS-triple $(M,D,\varphi)$ and let $L_{S}$ be the greatest Stone
subalgebra of $L$. Then $L$ is a perfect extension of $L_{S}$ if and
only if $M$ is a perfect extension of $B(M)$.
\end{cor}

From the above results it follows that in the considered special
case when a principal MS-algebra $L$ is a perfect extension of its
greatest Stone subalgebra $L_{S}$, our solution in
Theorem~\ref{thm:prob1} to Problem 1 can slightly be simplified by
not considering all principal congruences  $\alpha\in \Con(M)$ but
only special principal congruences $\alpha\in \Con(M)$ of the form
$\alpha = \theta (0,a) = \theta (a',1)$.

Our reasoning here is as follows. If $M$ is a perfect extension of
$B(M)$, then every congruence $\theta_1$ on $M$ is generated by some
congruence on $B(M)$, which is a join of principal congruences on
$B(M)$. Since $B(M)$ is Boolean, every principal congruence on
$B(M)$ is of form $\theta (0,a) = \theta (a',1)$. Hence, $\theta_1$
is a join of such congruences and by using the same argument as in
the proof of Theorem~\ref{thm:prob1} we obtain:

\begin{cor}\label{cor:prob1}
Let $M$ be a de Morgan algebra that is a perfect extension of $B(M)$, let $D$ be a bounded distributive
lattice, and let $A$ be a sublattice of $\Con(M)\times \Con(D)$.

Then  $A$ is representable if and only if
\begin{newlist}
\item[\rm(1)] $A$ is join-closed and
\item[\rm(2)] $A$ is down-closed in first coordinate, that is, $(\theta_1,\theta_2)\in A$ and $\alpha\leq\theta_{1}$ imply $(\alpha,\theta_{2})\in
A$;
\item[\rm(3)] there exists a lattice $(0,1)$-homomorphism $\varphi: M\to D$ such that for every principal congruence $\alpha\in
\Con(M)$
 of the form $\alpha = \theta (0,a) = \theta (a',1)$ and every $\theta_{2}\in \Con(D)$
\[
(\theta_{1},\theta_2) \in A \iff \Con (\varphi)(\theta_1) \subseteq \theta_2.
\]
\end{newlist}
\end{cor}

\section{Examples}\label{exam}

In the second part of the previous section we considered Problem~1
in the special case when a principal MS-algebra $L$ is a perfect
extension of its greatest Stone subalgebra $L_{S}$. When we
associate the principal MS-algebra $L$ with a principal MS-triple
$(M,D,\varphi)$, we know by Corollary~\ref{5.4} that this special
case occurs
when the de Morgan algebra $M$ is a perfect extension of its Boolean subalgebra $B(M)$. In this section we illustrate in examples 
when this condition is satisfied as well as when it is not
satisfied.

In Figure~\ref{fig:M1M2}
 we see two small examples of de Morgan algebras: $M_1$ is the four-element subdirectly irreducible de Morgan algebra and
$M_2$ is the de Morgan algebra $\bold 1\oplus {\bold 2}^3\oplus \bold 1$. 
For both $i=1, 2$ the Boolean subalgebra $B(M_i)=\{0,1\}$ is simple.
(The elements of $B(M_i)$ are in all figures shaded.)

\begin{figure}[ht]
\centering
\begin{tikzpicture}[scale=0.625]
  \begin{scope}
    \node[anchor=north] at (-1,0) {$M_1$};
    \node[shaded] (00) at (0,0) {};
    \node[unshaded] (10) at (2,2) {};
    \node[unshaded] (01) at (-2,2) {};
    \node[shaded] (11) at (0,4) {};
    \draw[order] (00) -- (10) -- (11);
    \draw[order] (00) -- (01) -- (11);
    \node[label,anchor=west] at (00) {$0 = 1^\circ$};
    \node[label,anchor=west] at (10) {$b=b^\circ$};
    \node[label,anchor=east] at (01) {$a=a^\circ$};
    \node[label,anchor=east] at (11) {$1 = 0^\circ$};
  \end{scope}
  \begin{scope}[xshift=8cm]
    \node[anchor=north] at (2,0) {$M_2$ with $\theta'$};
    \node[shaded] (0) at (0,0) {};
    \node[unshaded] (000) at (0,2) {};
    \node[unshaded] (100) at (2,4) {};
    \node[unshaded] (010) at (0,4) {};
    \node[unshaded] (001) at (-2,4) {};
    \node[unshaded] (110) at (2,6) {};
    \node[unshaded] (101) at (0,6) {};
    \node[unshaded] (011) at (-2,6) {};
    \node[unshaded] (111) at (0,8) {};
    \node[shaded] (1) at (0,10) {};
    \draw[order] (0) -- (000);
    \draw[order] (111) -- (1);
    \draw[order] (000) -- (100) -- (110) -- (111);
    \draw[order] (000) -- (001) -- (011) -- (111);
    \draw[order] (000) -- (010);
    \draw[order] (101) -- (111);
    \draw[order] (110) -- (010) -- (011);
    \draw[order] (001) -- (101) -- (100);
    \draw[curvy] (000) to [bend left] (010);
    \draw[curvy] (000) to [bend right] (010);
    \draw[curvy] (001) to [bend left] (011);
    \draw[curvy] (001) to [bend right] (011);
    \draw[curvy] (100) to [bend left] (110);
    \draw[curvy] (100) to [bend right] (110);
    \draw[curvy] (101) to [bend left] (111);
    \draw[curvy] (101) to [bend right] (111);
    \node[label,anchor=east] at (0) {$0 = 1^\circ$};
    \node[label,anchor=west] at (000) {$a$};
    \node[label,anchor=east] at (001) {$b$};
    \node[label,anchor=west] at (010) {$c$};
    \node[label,anchor=west] at (100) {$d$};
    \node[label,anchor=east] at (011) {$b^\circ$};
    \node[label,anchor=west] at (101) {$c^\circ$};
    \node[label,anchor=west] at (110) {$d^\circ$};
    \node[label,anchor=west] at (111) {$a^\circ$};
    \node[label,anchor=east] at (1) {$1=0^\circ$};
    \end{scope}
\end{tikzpicture}
\caption{The de Morgan algebras $M_1$ and $M_2$.}\label{fig:M1M2}
\end{figure}
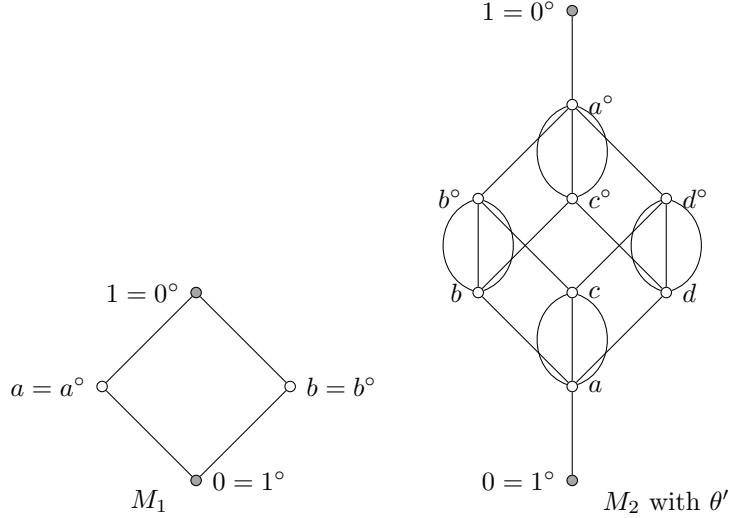

Since the four-element de Morgan algebra $M_1$ is simple, too, it is
automatically a perfect extension of $B(M_1)$. So the MS-algebra
$L_1$ taken from \cite[Example 3.3]{BGH11} and depicted in
Figure~\ref{fig:L1S} is a perfect extension of its greatest Stone
subalgebra $S:= (L_1)_S$ which is the three-element chain. Here
$L_1$ is created from the principal MS-triple $(M_1,D,\varphi_1)$
where $D=\{0,1\}$ is the two-element lattice and $\varphi_1: M_1 \to
D$ is given by $\varphi_1 (0) = \varphi_1 (a) = 0, \varphi_1 (b) =
\varphi_1 (1) = 1$. It is easy to see that
 $\theta = \triangle_{L_1}\cup {D(L_1)}^2\cup \{(b,0),(b,1)\}^2$ is the only non-trivial congruence of $L_1$ which is the unique
 extension
of the congruence $\theta_S = \triangle_{S}\cup {D(S)}^2$.

\begin{figure}[ht]
\centering
\begin{tikzpicture}[scale=0.625]
  \begin{scope}
    \node[anchor=north] at (-1,0) {$L_1$};
    \node[shaded] (00) at (0,0) {};
    \node[unshaded] (10) at (2,2) {};
    \node[unshaded] (01) at (-2,2) {};
    \node[unshaded] (11) at (0,4) {};
    \node[unshaded] (20) at (4,4) {};
    \node[shaded] (22) at (2,6) {};
    \draw[order] (00) -- (01) -- (11) -- (22);
    \draw[order] (00) -- (10) -- (11);
    \draw[order] (10) -- (20) -- (22);
    \node[label,anchor=west] at (00) {$(0,0)$};
    \node[label,anchor=west] at (10) {$(b,0)$};
    \node[label,anchor=east] at (01) {$(a,0)$};
    \node[label,anchor=east] at (11) {$(1,0)$};
    \node[label,anchor=west] at (20) {$(b,1)$};
    \node[label,anchor=west] at (22) {$(1,1)$};
  \end{scope}
  \begin{scope}[xshift=10cm]
    \node[anchor=north] at (2,0) {$S$ with $\theta_S$};
    \node[shaded] (00) at (0,0) {};
    \node[unshaded] (11) at (0,3) {};
    \node[shaded] (22) at (0,6) {};
     \draw[order] (00) -- (11) -- (22);
     \draw[curvy] (11) to [bend left] (22);
    \draw[curvy] (11) to [bend right] (22);
    \node[label,anchor=east] at (00) {$(0,0)$};
    \node[label,anchor=east] at (11) {$(1,0)$};
    \node[label,anchor=east] at (22) {$(1,1)$};
       \end{scope}
\end{tikzpicture}
\caption{The MS algebra $L_1$ and its Stone subalgebra $S$.}\label{fig:L1S}
\end{figure}
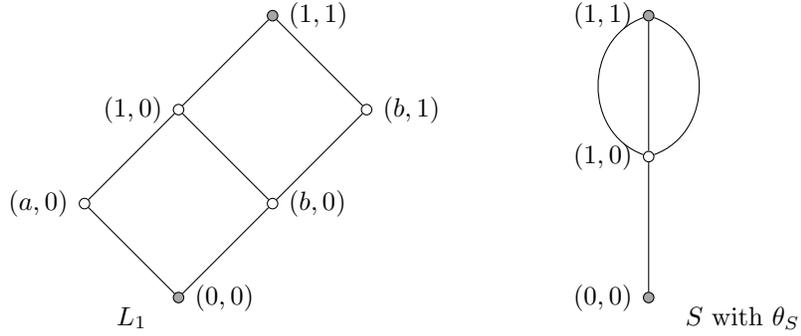

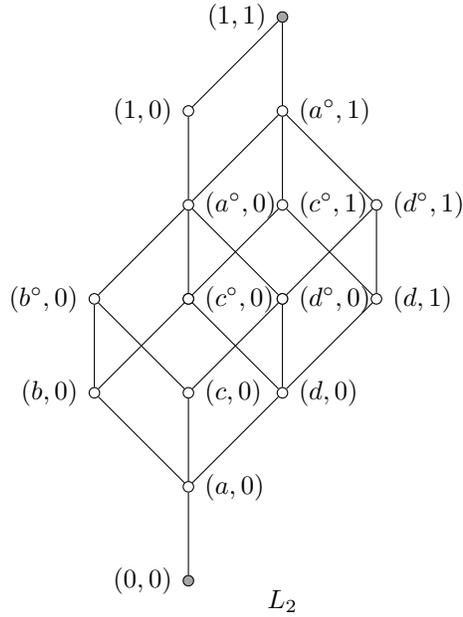
\begin{figure}[ht]
\centering
\begin{tikzpicture}[scale=0.625]
  \begin{scope}
    \node[anchor=north] at (2,0) {$L_2$};
    \node[shaded] (0) at (0,0) {};
    \node[unshaded] (000) at (0,2) {};
    \node[unshaded] (100) at (2,4) {};
    \node[unshaded] (010) at (0,4) {};
    \node[unshaded] (001) at (-2,4) {};
    \node[unshaded] (110) at (2,6) {};
    \node[unshaded] (101) at (0,6) {};
    \node[unshaded] (011) at (-2,6) {};
    \node[unshaded] (111) at (0,8) {};
    \node[unshaded] (1) at (0,10) {};
   \node[unshaded] (1001) at (4,6) {};
    \node[unshaded] (110) at (2,6) {};
     \node[unshaded] (1101) at (4,8) {};
    \node[unshaded] (101) at (0,6) {};
   \node[unshaded] (1011) at (2,8) {};
      \node[unshaded] (1111) at (2,10) {};
       \node[shaded] (11) at (2,12) {};
    \draw[order] (0) -- (000);
    \draw[order] (111) -- (1);
    \draw[order] (000) -- (100) -- (110) -- (111);
    \draw[order] (000) -- (001) -- (011) -- (111);
    \draw[order] (000) -- (010);
    \draw[order] (101) -- (111);
    \draw[order] (110) -- (010) -- (011);
    \draw[order] (001) -- (101) -- (100);
 \draw[order] (100) -- (1001);
\draw[order] (101) -- (1011);
\draw[order] (110) -- (1101);
    \draw[order] (111) -- (1111);
    \draw[order] (1001) -- (1011);
    \draw[order] (1001) -- (1101);
    \draw[order] (1101) -- (1111);
 \draw[order] (1011) -- (1111);
\draw[order] (1111) -- (11);
 \draw[order] (1) -- (11);
    \node[label,anchor=east] at (0) {$(0,0)$};
    \node[label,anchor=west] at (000) {$(a,0)$};
    \node[label,anchor=east] at (001) {$(b,0)$};
    \node[label,anchor=west] at (010) {$(c,0)$};
    \node[label,anchor=west] at (100) {$(d,0)$};
    \node[label,anchor=east] at (011) {$(b^\circ,0)$};
    \node[label,anchor=west] at (101) {$(c^\circ,0)$};
    \node[label,anchor=west] at (110) {$(d^\circ,0)$};
    \node[label,anchor=west] at (111) {$(a^\circ,0)$};
    \node[label,anchor=east] at (1) {$(1,0)$};
  \node[label,anchor=west] at (1001) {$(d,1)$};
    \node[label,anchor=west] at (1011) {$(c^\circ,1)$};
    \node[label,anchor=west] at (1101) {$(d^\circ,1)$};
    \node[label,anchor=west] at (1111) {$(a^\circ,1)$};
    \node[label,anchor=east] at (11) {$(1,1)$};
    \end{scope}
\end{tikzpicture}
\caption{The MS algebra $L_2$}\label{fig:L2}
\end{figure}

On the other hand, the ten-element de Morgan algebra $M_2$ is not
simple and so it is not a perfect extension of $B(M_2)$. A
non-trivial congruence $\theta'$ of $M_2$ is depicted in
Figure~\ref{fig:M1M2}.

For the principal MS-triple $(M_2,D,\varphi_2)$ where $D$ is the
same two-element lattice as above and $\varphi: M_2 \to D$ is given
by
\begin{align}
&\varphi_2(0)=\varphi_2(a)=\ \varphi_2(b)\ = \varphi_2(c)\ =  \varphi_2(b^\circ)=0,\notag\\
&\varphi_2(d)=\varphi_2(c^\circ)=\varphi_2(d^\circ)=\varphi_2(a^\circ)=\varphi_2(1)=1\notag
\end{align}
we obtain the principal MS-algebra $L_2$ in Figure~\ref{fig:L2}. By
Corollary~\ref{5.2}, $L_2$ is not a perfect extension of its
greatest Stone subalgebra $(L_2)_{S}$ which is again the
three-element chain $S =\{(0,0),(1,0),(1,1)\}$. It is easy to check
that the only non-trivial congruence $\theta_S = \triangle_{S}\cup
{D(S)}^2$ of $S$ has at least two different extensions in
$\Con(L_2)$, namely the congruences
\begin{align}
\theta_1 = \triangle_{L_2}
&\cup {D(L_2)}^2 \cup \{(d,0),(d,1)\}^2\cup\{(d^{\circ},0),(d^{\circ},1)\}^{2}\cup\{(c^{\circ},0),(c^{\circ},1)\}^{2}\notag\\
&\cup\{(a^{\circ},0),(a^{\circ},1)\}^{2},\notag\\
\theta_2 = \triangle_{L_2}
&\cup {D(L_2)}^2\cup \{(d,0),(d,1),(d^{\circ},0),(d^{\circ},1)\}^{2}\cup \{(a,0),(c,0)\}^{2}\notag\\
&\cup\{(b,0),(b^{\circ},0)\}^{2}\cup\{(c^{\circ},0),(c^{\circ},1),(a^{\circ},0),(a^{\circ},1)\}^{2}.\notag
\end{align}

\bibliographystyle{amsplain}
\renewcommand{\bibname}{References}

\end{document}